\def\misajour{28 ao\^{u}t 2011}

\documentclass[10pt, leqno]{article}
\usepackage{amsmath,amsthm}
\usepackage{amssymb,latexsym}
\usepackage{enumerate}

\usepackage{hyperref} 
\usepackage[applemac]{inputenc}

\newtheorem{thm}{Théorème}[section]
\newtheorem{cor}[thm]{Corollaire}
\newtheorem{lem}[thm]{Lemme}
\newtheorem{prop}[thm]{Proposition}

\theoremstyle{definition}

\newtheorem*{defin}{D\'efinition}

\numberwithin{equation}{section}

\def\C{\mathbf{C}} 
\def\N{\mathbf{N}} 

\def\mathbfp{\boldsymbol{p}}
\def\mathbfptilde{\widetilde{\mathbfp}}
 
\def\P{\mathbf{P}} 
\def\Q{\mathbf{Q}} 
\def\bR{\mathbf{R}} 
\def\Z{\mathbf{Z}}
\def\denominateur{q}

\def\alphatilde{\widetilde{\alpha}}
\def\etatilde{\widetilde{\eta}}
\def\uk{\underline{k}}
\def\uu{\underline{u}}
\def\uv{\underline{v}}
\def\udelta{\underline{\delta}}
\def\usigma{\underline{\sigma}}
\def\ugamma{\underline{\gamma}}
\def\uvarepsilon{\underline{\varepsilon}}
\def\OK{{\mathcal{O}_K}}
\def\OS{O_S}
\def\OKtimes{{\mathcal{O}_K^\times}}
\def\calB{{\mathcal{B}}}
\def\calE{{\mathcal{E}}}
\def\calF{{\mathcal{F}}}
\def\calP{{\mathcal{P}}}
\def\calT{{\mathcal{T}}}
\def\SSomme{{\calT}}
 \def\gothA{\mathfrak {A}}
 \def\gothB{\mathfrak {B}}
 \def\gothP{\mathfrak {P}}
\def\ord{{\mathrm ord}}
\def\rmh{{\mathrm h}} 
\def\rmN{{\mathrm N}}

\def\Card{{\mathrm {Card}}}

\def\sgn{{\mathrm{sgn}}}
\def\pgcd{\mathrm{pgcd}}

\begin{document}





\title{Familles d'équations de Thue--Mahler \\
n'ayant que des solutions triviales}

\author{Claude Levesque\\
Département de mathématiques et de statistique\\
 Université Laval\\
Québec (Québec)\\
CANADA G1V 0A6\\
E-mail: Claude.Levesque@mat.ulaval.ca
\and 
Michel Waldschmidt\\\
Institut de Mathématiques de Jussieu\\
Université Pierre et Marie Curie (Paris 6)\\
4 Place Jussieu\\
F -- 75252 PARIS Cedex 05, FRANCE\\
E-mail: miw@math.jussieu.fr}

\date{\misajour}

\maketitle


\renewcommand{\thefootnote}{}

\footnote{2010 \emph{Mathematics Subject Classification}: Primary 
11D59 
 ; Secondary 
11D45 
11D61 
11D25 
11J87} 
\footnote{\emph{Key words and phrases}: Diophantine equations,
	Families of Thue-Mahler equations,
	units of algebraic number fields,
	Schmidt subspace theorem.}
\null \hfill
{\it En hommage à André Schinzel.}

\renewcommand{\thefootnote}{\arabic{footnote}}
\setcounter{footnote}{0}


\begin{abstract} 
 Let $K$ be a number field, let $S$ be a finite set of places of $K$ containing the archimedean places and let $\mu$, $\alpha_1,\alpha_2,\alpha_3$ be non--zero elements in $K$. Denote by $\OS$ the ring of $S$--integers in $K$ and by $\OS^\times$ the group of $S$--units. Then the set of equivalence classes (namely, up to multiplication by $S$--units) of the solutions 
$(x,y,z,\varepsilon_1, \varepsilon_2,\varepsilon_3,\varepsilon)\in\OS^3\times(\OS^\times)^4$ 
of the diophantine equation 
$$ (X-\alpha_1 E_1 Y) (X-\alpha_2E_2 Y) (X-\alpha_3E_3 Y)Z=\mu E,
$$ 
satisfying 
$\Card\{\alpha_1\varepsilon_1,\alpha_2\varepsilon_2,\alpha_3\varepsilon_3\}= 3$, is finite. 
With the help of this last result, we exhibit, for every integer $n>2$, new families of Thue-Mahler equations of degreee $n$  having only trivial solutions.
Furthermore, we produce an effective upper bound for the number of these solutions. The proofs of this paper rest heavily on Schmidt's subspace theorem.
 
\begin{center}
{\bf Résumé}
\end{center} 

 Soit $K$ un corps de nombres, soit 
$S$ un ensemble fini de places de $K$ contenant les places archimédiennes et soient $\mu$, $\alpha_1,\alpha_2,\alpha_3$ des éléments non nuls de $K$. Notons $\OS$ l'anneau des $S$--entiers de $K$ et $\OS^\times$ le groupe des $S$--unités. Alors l'ensemble des classes d'équivalence (c'est-à-dire, à multiplication près par des $S$--unités) des solutions 
$(x,y,z,\varepsilon_1, \varepsilon_2,\varepsilon_3,\varepsilon)\in\OS^3\times(\OS^\times)^4$
de l'équation diophantienne 
 $$
 (X-\alpha_1 E_1 Y) (X-\alpha_2E_2 Y) (X-\alpha_3E_3 Y)Z=\mu E,
 $$ 
 vérifiant 
$\Card\{\alpha_1\varepsilon_1,\alpha_2\varepsilon_2,\alpha_3\varepsilon_3\}= 3$, est fini. Grâce à ce dernier résultat, nous pouvons exhiber pour chaque entier $n\ge 3$ de nouvelles familles
d'équations de Thue-Mahler de degré $n$ ne possédant que des  solutions triviales. De plus, nous donnons une borne supérieure explicite pour le nombre de ces solutions. 
Les démonstrations de cet article reposent sur le théorème du sous--espace de Schmidt. 

\end{abstract}

\section{Introduction}

Les premières familles d'équations de Thue n'ayant que des solutions dites triviales ont été données par E. Thomas en 1990 \cite{MR1042497}. D'autres familles ont été trouvées par la suite \cite{ClemensHeuberger}. Les familles d'équations de Thue étudiées jusqu’à présent font presque toutes intervenir des équations variant 
avec un paramètre dont dépend un corps de nombres attaché à ladite équation. 
À date, il n’y avait pas encore de familles d’équations de Thue–\-Mahler pour lesquelles la finitude du nombre de solutions ait été établie.
Notre résultat va permettre d'attacher, à chaque corps de nombres algébriques possédant une infinité d'unités, une famille infinie d'équations de Thue--\-Mahler n'ayant qu'un nombre fini de solutions non--triviales. L'outil diophantien que nous allons mettre en {\oe}uvre est le théorème du sous--espace de W.M.~Schmidt \cite{MR1359604,AV,MR883451}. Nous poursuivrons ultérieurement ce travail en remplaçant le théorème non effectif du sous--espace par des minorations de formes linéaires de logarithmes, de façon à rendre effectifs nos résultats, au moins dans certains cas particuliers.

\section{Résultats connus}\label{S:ResultatsConnus}

Soient $K$ un corps de nombres et $S$ un ensemble fini de places de $K$ contenant les places archimédiennes. Nous notons $\OK$ l'anneau des entiers algébriques de $K$, $\OKtimes$ le groupe des unités de $K$ (éléments inversibles de $\OK$). De plus, pour $S$ ensemble fini de places de $K$ contenant les places archimédiennes,  $\OS$ est  l'anneau des $S$--entiers de $K$, à savoir l'ensemble des éléments $x$ de $K$ tels que $|x|_v\le 1$  pour toute place   $v \not\in S$.  De plus,  $\OS^\times$ est le groupe des $S$--unités, c'est-à-dire des éléments inversibles de $\OS$,  donc des éléments $x$ de $K$ tels que $|x|_v = 1$ pour tout $v \not\in S$.

Soient $\mu$, $\alpha_1,\ldots,\alpha_n$ des éléments non nuls de $K$. On considère pour commencer l'équation diophantienne de Thue--Mahler
\begin{equation}\label{Equation:ThueMahler}
 (X-\alpha_1 Y)\cdots (X-\alpha_nY)=\mu E,
\end{equation}
où les valeurs prises par les inconnues $X$ et $Y$ sont des éléments $x$ et $y$ de $\OS$ vérifiant $xy\not=0$, et où les valeurs prises par l'inconnue $E$ sont des éléments $\varepsilon$ de $\OS^\times$. 

Si $(x,y,\varepsilon)$ est une solution de l'équation $(\ref{Equation:ThueMahler})$, alors pour tout $\eta\in\OS^\times$ le triplet $(\eta x,\eta y,\eta^n\varepsilon)$ est aussi une solution. Selon \cite{MR1004133}, \S2, deux telles solutions sont dites {\it $S$--dépendantes. } Autrement, elles sont dites $S$--{\it indépendantes}. En voici une formulation équivalente.

\begin{defin}
Deux solutions $(x,y,\varepsilon)$ et $(x',y',\varepsilon')$ dans $\OS^2\times\OS^\times$ de l'équation $(\ref{Equation:ThueMahler})$ sont dites {\it $S$--dépendantes} si les points de $\P^1(K)$ de coordonnées projectives $(x:y)$ et $(x',y')$ sont les mêmes. Autrement, elles sont dites $S$--{\it indépendantes}.
\end{defin}

Le résultat suivant a été démontré par Parry en 1950 à la suite de travaux de Thue et Mahler notamment (voir par exemple \cite{MR1004133}, \S2 et
 \cite{MR88h:11002}, théorème 5.4).

 \begin{thm}[Thue--Mahler--Parry]\label{Theoreme:ThueMahler}
Sous l'hypothèse
$$
\Card\{\alpha_1,\ldots,\alpha_n\}\ge 3,
$$
le nombre maximum de solutions deux à deux $S$--indépendantes $(x,y,\varepsilon)\in\OS\times\OS\times\OS^\times$ avec $xy\not=0$ de 
l'équation $(\ref{Equation:ThueMahler})$ est fini. 
 
 \end{thm}

Erd{\H{o}}s,   Stewart et Tijdeman   \cite{MR937987} ont donné des exemples dans lesquels le nombre de solutions de  l'équation  $(\ref{Equation:ThueMahler})$ est remarquablement grand.

Une conséquence du théorème $\ref{Theoreme:ThueMahler}$ est la suivante. 

\begin{cor}\label{Corollaire:ThueMahler}
Soit $f(X,Y)\in\Z[X,Y]$ une forme binaire à coefficients dans $\Z$ ayant au moins trois facteurs linéaires deux à deux non proportionnels dans sa factorisation sur $\C$. Soit $k\in\Z\setminus\{0\}$ et soient $p_1,\ldots,p_t$ des nombres premiers. Alors l'équation $f(X,Y)=kp_1^{Z_1}\cdots p_t^{Z_t}$ n'a qu'un nombre fini de solutions $(x,y,z_1,\ldots,z_t)$ dans $\Z^2\times \N^t$ vérifiant $xy\not=0$ et $\pgcd(xy, p_1\cdots p_t)=1$. 
\end{cor}

Nous utiliserons la généralisation $\rmN_S$ de la norme, introduite par exemple dans \cite{MR1004133} et dont nous rappelons la définition au début du \S$\ref{S:Outils diophantiens}$. 
Dans \cite{MR1004133}, \S2, Evertse et Gy\H{o}ry font intervenir deux relations d'équivalence : la première concerne les solutions des inéquations de Thue--Mahler
\begin{equation}\label{Inequation:ThueMahler}
0<\rmN_S(f(x,y)) \leq m.
\end{equation}
Selon  \cite{MR1004133}, deux solutions $(x,y)$, $(x',y')$ de  l'inéquation $(\ref{Inequation:ThueMahler})$ sont dépendantes s'il existe $\eta\in K^\times$ tel que $x'=\eta x$ et $y'=\eta y$. 
Voici la seconde, qui concerne les formes binaires de $\OS[X,Y]$.
 
 \begin{defin}
Deux formes binaires $f(X,Y)$ et $g(X,Y)$ de $\OS[X,Y]$, sont dites {\it $S$--équivalentes}
s'il existe $\alpha, \beta, \gamma, \delta$ dans $\OS$ et $\eta $ dans $\OS^\times$ v\'erifiant $\alpha\delta-\beta\gamma \in\OS^\times$ et $g(X,Y) = \eta 
f(\alpha X+\beta Y, \gamma X+\delta Y)$. 
\end{defin}


Soient $K$ un corps de nombres, $S$ un ensemble fini de places de $K$ contenant les places archimédiennes, et $n$ un entier $\ge 3$. Notons $\calF(n, K,S)$ l'ensemble des formes binaires $f\in \OS[X,Y]$ de degré $n$ qui se d\'ecomposent en facteurs lin\'eaires dans $K[X,Y]$ et dont la décomposition contient au moins trois facteurs linéaires distincts. Voici la première partie du théorème 2 de \cite{MR1004133}. 

\begin{thm}[Evertse et Gy\H{o}ry]\label{Theoreme:Evertse}
Avec les notations ci--dessus, pour tout entier $m>0$, il n'y a qu'un nombre fini de classes de $S$--équivalence de formes binaires $f$ dans $\calF(n, K,S)$ pour lesquelles il existe plus de deux solutions  indépendantes 
$(x,y)\in \OS\times\OS$ à l'inéquation $(\ref{Inequation:ThueMahler})$.

\end{thm}

Les autres résultats de \cite{MR1004133} font intervenir une extension finie $L$ de $K$, mais pour le théorème $\ref{Theoreme:Evertse}$ cela n'apporte pas plus de généralité. 

\section{Énoncés de nos résultats}\label{S:Enonces}
Notre résultat principal (théorème $\ref{Theoreme:Principal})$ consiste à remplacer d'une part trois des facteurs $X-\alpha_i Y$ (disons pour $i=1,2,3$) de l'équation $(\ref{Equation:ThueMahler})$ par des facteurs de la forme $X-\alpha_iE_iY$, où $E_1,E_2,E_3$ sont des inconnues supplémentaires prenant des valeurs dans $\OS^\times$, et à remplacer d'autre part le produit des autres facteurs par $Z$ où l'inconnue $Z$ prend ses valeurs dans l'anneau des $S$--entiers. 
 On considère donc l'équation plus générale
 \begin{equation}\label{Equation:ThueGeneralisee}
 (X-\alpha_1 E_1 Y) (X-\alpha_2 E_2 Y) (X-\alpha_3E_3 Y)Z=\mu E,
\end{equation}
 où les inconnues $X$, $Y$ et $Z$ sont à valeurs dans $\OS$ et où les inconnues $E_1,E_2, E_3,E$ sont à valeurs dans $\OS^\times$. On note $(X,Y,Z,E_1,E_2, E_3,E)$
 le septuplet formé par les inconnues, et $(x,y,z,\varepsilon_1,\varepsilon_2,\varepsilon_3,\varepsilon)$
 les septuplets formés par les solutions dans $\OS^3\times(\OS^\times)^4$. 
 
 \begin{defin}
Nous appellerons {\it triviales} les solutions pour lesquelles $xy=0$. 
\end{defin}

Nous supposerons que les solutions satisfont 
$$ 
\Card\{\alpha_1\varepsilon_1, \alpha_2\varepsilon_2, \alpha_3\varepsilon_3\}= 3.
$$ 
Si $(x,y,z,\varepsilon_1,\varepsilon_2,\varepsilon_3,\varepsilon)$ est une solution non triviale de $(\ref{Equation:ThueGeneralisee})$, alors, pour tout $\eta\in\OS^\times$, les septuplets 
$$
 (\eta x,\eta y,z,\varepsilon_1,\varepsilon_2,\varepsilon_3,\eta^3\varepsilon),
 \quad 
 (x,y,\eta z,\varepsilon_1,\varepsilon_2,\varepsilon_3,\eta \varepsilon) 
\quad\hbox{et}\quad 
 (x,\eta^{-1} y, z,\eta\varepsilon_1,\eta\varepsilon_2,\eta\varepsilon_3, \varepsilon)
$$ 
sont aussi des solutions non triviales de $(\ref{Equation:ThueGeneralisee})$, que l'on qualifiera de {\it $S^3$--dépendantes} de la solution initiale.

\begin{defin}  
Deux solutions non triviales $(x,y,z,\varepsilon_1,\varepsilon_2,\varepsilon_3,\varepsilon)$ et $(x',y',z',\varepsilon'_1,\varepsilon'_2,\varepsilon'_3,\varepsilon')$ de l'équation $(\ref{Equation:ThueGeneralisee})$ seront dites {\it $S^3$--dépendantes} s'il existe des $S$--unités $\eta_1$, $\eta_2$ et $\eta_3$ dans $\OS^\times$ telles que 
$$ 
x'=x\eta_1,\; y'=y\eta_1\eta_3^{-1}, \; z'=z\eta_2, \; \varepsilon'_i=\varepsilon_i\eta_3 \quad(i=1,2,3),\; 
 \varepsilon'=\varepsilon \eta_1^3\eta_2.
$$
Sinon, ces deux solutions seront dites $S^3$--{\it indépendantes}.
\end{defin}

Ceci définit une relation d'équivalence sur l'ensemble des solutions non triviales de $(\ref{Equation:ThueGeneralisee})$. Quand le groupe des $S$--unités $\OS^\times$ a un rang $\ge 1$, chaque classe de $S^3$--dépendance de solutions non triviales contient une infinité d'éléments.

Le principal but de cet article est de démontrer le résultat suivant, qui généralise le théorème $\ref{Theoreme:ThueMahler}$.
Il s'agit apparemment du premier énoncé donnant la finitude du nombre de solutions de familles d'équations de Thue--Mahler.

\begin{thm}\label{Theoreme:Principal}
Soient $K$ un corps de nombres, $S$ un ensemble fini de places de $K$ contenant les places archimédiennes, $\mu$, $\alpha_1,\alpha_2,\alpha_3$ des éléments non nuls de $K$. Alors l'ensemble des classes de $S^3$--dépendance 
des solutions non triviales $(x,y,z,\varepsilon_1,\varepsilon_2,\varepsilon_3,\varepsilon)$ dans $\OS^3\times(\OS^\times)^4$ de l'équation 
$$
 (X-\alpha_1 E_1 Y) (X-\alpha_2 E_2 Y) (X-\alpha_3E_3 Y)Z=\mu E,
$$
satisfaisant la condition $\Card\{\alpha_1\varepsilon_1, \alpha_2\varepsilon_2, \alpha_3\varepsilon_3\}= 3$, est fini. 
Le nombre de ces classes est majoré par une constante  $\kappa_1$, 
ne dépendant que 
de $K$, $\rmN_S(\mu)$, $\alpha_1, \alpha_2,\alpha_3$ et du rang $s$ du groupe $\OS^\times$, et 
dont la valeur est explicitée dans la formule $(\ref{Equation:cstTheoreme})$.
\end{thm}

Quitte à agrandir $S$, on peut supposer $\alpha_1=\alpha_2=\alpha_3=\mu=1$ (cela augmente le rang de $\OS^\times$, il faudrait en tenir compte si on voulait donner une borne explicite). Dans ce cas, les variables $E$ et $Z$ sont redondantes (prendre $\eta_2=z^{-1}$ dans la définition de la $S^3$--dépendance pour se ramener à $z=1$). De même, il n'y a pas de restriction à supposer $E_3=1$ (prendre $\eta_3=\varepsilon_3^{-1}$). L'équation que l'on considère se ramène donc à 
\begin{equation}\label{Equation:5variables}
 (X- Y) (X- E_1 Y) (X- E_2 Y)=E,
\end{equation}
en $5$ inconnues, $(X,Y,E_1,E_2,E)$. La relation d'équivalence entre les quintuplets dans $\OS^2\times(\OS^\times)^3$ formés de solutions est essentiellement celle du \S\ref{S:ResultatsConnus}: nous dirons encore que deux solutions $(x,y,\varepsilon_1,\varepsilon_2,\varepsilon)$ et $(x',y',\varepsilon'_1,\varepsilon'_2,\varepsilon')$ de $(\ref{Equation:5variables})$ sont {\it $S$--dépendantes} s'il existe $\eta\in\OS^\times$ tel que
$$
x'=x\eta,\; y'=y\eta , \; \varepsilon'_1=\varepsilon_1, \; \varepsilon'_2=\varepsilon_2, \;
 \varepsilon'=\varepsilon \eta^3. 
$$
L'énoncé suivant est donc équivalent au théorème \ref{Theoreme:Principal} (à la valeur près de la borne explicite pour le nombre de classes de solutions, que nous ne précisons pas). 

\begin{thm}\label{thm:5variables}
Le nombre de classes de $S$--dépendance de solutions de l'équation $(\ref{Equation:5variables})$ est fini. 
\end{thm}

Il est intéressant de noter que nous avons ramené la question de la finitude du nombre de solutions de l'équation de Thue--Mahler en degré quelconque à celle d'une équation cubique (voir à ce sujet \cite{CL-MW-Vietnam}).

Pour faire le lien avec le théorème $\ref{Theoreme:Evertse}$, montrons que notre théorème $\ref{Theoreme:Principal}$ permet d'établir l'existence d'une infinité 
de classes de $S$--équivalence de formes binaires de $\calF(n, K,S)$ produisant des inéquations $(\ref{Inequation:ThueMahler})$ dont toutes les solutions sont triviales. 

Du théorème $\ref{thm:5variables}$ on déduit facilement l'énoncé suivant. 

\begin{thm}\label{thm:Thue}
Soient $K$ un corps de nombres, $S$ un ensemble fini de places de $K$ contenant les places archimédiennes, $n$ un entier $\ge 3$, $\alpha_1,\ldots,\alpha_n$ des éléments de $K^\times$ et $f\in K[X,Y]$ la forme binaire 
$$
f(X,Y)=\bigl(X- \alpha_1 Y\bigr) \bigl(X- \alpha_2 Y\bigr)\cdots \bigl(X- \alpha_n Y\bigr).
$$
Pour $\uvarepsilon=(\varepsilon_1,\ldots,\varepsilon_n)\in (\OS^\times)^n$, on note $f_{\uvarepsilon} \in K[X,Y]$ la forme binaire 
$$
f_{\uvarepsilon} (X,Y)= (X-\alpha_1 \varepsilon_1 Y) (X-\alpha_2 \varepsilon_2 Y)\cdots (X-\alpha_n\varepsilon_n Y).
$$
On désigne par $\calE$ l'ensemble des éléments $\uvarepsilon$ de $(\OS^\times)^n$ tels que $\varepsilon_1=1$ et $\Card\{\alpha_1\varepsilon_1,\alpha_2\varepsilon_2,\ldots,\alpha_n\varepsilon_n\}\ge3$. 
Alors il existe un sous--ensemble fini $\calE^\star$ de $\calE$ tel que, pour tout $\uvarepsilon\in\calE\setminus \calE^\star$ et pour tout $(x,y)\in\OS\times\OS$, la condition 
$$
 f_{\uvarepsilon} (x,y) \in\OS^\times
$$
implique $xy=0$.
\end{thm}

Ce théorème $\ref{thm:Thue}$ se déduit du   théorème $\ref{Theoreme:Principal}$ avec $\mu=1$: quitte à renuméroter les racines de $f_{\uvarepsilon}$,
on peut supposer $\Card\{\alpha_1\varepsilon_1,\alpha_2\varepsilon_2,\alpha_3\varepsilon_3\}=3$; on prend alors pour $z$ le produit $ (x-\alpha_4 \varepsilon_4 y) \cdots (x-\alpha_n\varepsilon_n y)$.

 Au \S$\ref{S:DemonstrationProposition}$, on prouvera la proposition suivante.

\begin{prop} \label{Proposition:InfiniteClassesEquivalence}
Soit $\uvarepsilon\in \calE\setminus\calE^\star$. Il n'y a qu'un nombre fini de $\uvarepsilon'\in \calE\setminus\calE^\star$ tel que les deux formes binaires $f_{\uvarepsilon}$ et $f_{\uvarepsilon'}$ soient $S$--équivalentes. 
\end{prop}

On déduit  encore du  théorème $\ref{Theoreme:Principal}$ le corollaire suivant. 

\begin{cor}\label{cor:inequationThueMahler}
Pour chaque $m\in \N$ et pour chaque $\uvarepsilon \in \calE $ en dehors d'un sous--ensemble fini (dépendant de $m$), l'inéquation 
$$
0<\rmN_S\bigl( f_{\uvarepsilon} (x,y)\bigr)\le m
$$
n'a pas de solution $(x,y)\in\OS\times\OS$ vérifiant $xy\not=0$. 
\end{cor}
 
 C'est ainsi que nous obtenons une infinité de classes de $S$--équivalence de formes binaires donnant lieu à des inéquations de Thue--Mahler $(\ref{Inequation:ThueMahler})$ dont toutes les solutions sont triviales. 

Pour énoncer le corollaire $\ref{Corollaire:Principal}$ ci--dessous, nous définissons une famille de formes binaires $f_{\varepsilon}$, indexée par les $S$--unités d'un corps de nombres, de la façon suivante.
Partons d'une forme
$$
f(X,Y) =a_0X^d+a_1X^{d-1}Y+\cdots+a_{d-1}Y^{d-1}X +a_dY^d \in\Z[X,Y]
$$ 
irréductible sur $\Z$ et écrivons 
$$
f(X,Y)=a_0\bigl(X-\sigma_1(\alpha) Y\bigr) \bigl(X-\sigma_2(\alpha) Y\bigr)\cdots \bigl(X-\sigma_d(\alpha) Y\bigr),
$$
où $\alpha$ est une racine de $f(X,1)$ et où $\sigma_1, \sigma_2, \dots,\sigma_d$ sont les $d$ plongements du corps $K:=\Q(\alpha)$ dans $\C$. 
Soient $p_1,\ldots,p_t$ des nombres premiers et $S$ l'ensemble des places du corps $K$ constitué des places archimédiennes et des places de $K$ au--dessus de $p_1,\ldots,p_t$. Tordons $f(X,Y)$ par une $S$--unité $\varepsilon \in \OS^\times$ en définissant la forme binaire $f_\varepsilon(X,Y)\in \Z[X,Y]$ de degré $d$ par 
$$
f_\varepsilon(X,Y)=a_0\bigl(X-\sigma_1(\alpha\varepsilon) Y\bigr) \bigl(X-\sigma_2(\alpha\varepsilon) Y\bigr) \cdots \bigl(X-\sigma_d(\alpha\varepsilon) Y\bigr).
$$
Du théorème $\ref{Theoreme:Principal}$ nous déduisons la généralisation suivante du corollaire $\ref{Corollaire:ThueMahler}$. 

\begin{cor}\label{Corollaire:Principal}
Soit $k\in\Z\setminus\{0\}$. 
 L'ensemble des $(x,y,\varepsilon,z_1,\ldots,z_t)$ dans $\Z^2\times \OS^\times\times  \N^t$ satisfaisant
\begin{equation}\label{Equation:CorollairePrincipal}
 f_\varepsilon(x,y)=kp_1^{z_1}\cdots p_t^{z_t}, 
 \end{equation}
 avec $xy\not=0$, $\pgcd(xy,p_1\cdots p_t)=1$ et $ [\Q(\alpha \varepsilon):\Q]\ge 3 $,
 est fini. Le nombre de ces solutions est majoré par  
 une constante $ \kappa_2$,
effectivement calculable en fonction de $\alpha$, $k$, $f$ et $t$, 
dont la valeur est explicitée dans la formule $(\ref{Equation:cstCor})$.

\end{cor}

En particulier, lorsque $ [\Q(\alpha \varepsilon):\Q]\ge 3$, l'équation de Thue 
$f_\varepsilon(X,Y) =\pm 1$, sauf pour un nombre fini de $S$--unités $\varepsilon$, n'a que des solutions triviales.
Un autre cas particulier du corollaire $\ref{Corollaire:Principal}$ s'énonce de la façon suivante: soient $K$ un corps de nombres de degré $\geq 3$, $k$ un entier non nul et $p_1, p_2, \dots , p_t$ des nombres premiers. 
Pour toute unité $\varepsilon$ de $K$, désignons par $g_{\varepsilon}(X,Y) \in {\bf Z}[X,Y]$ la version homog\`ene du
polyn\^ome minimal $g_{\varepsilon}(X,1) \in {\bf Z}[X]$ de $\varepsilon$. {\it Alors l'ensemble des unit\'es $\varepsilon$ de $K$, telles que 
$[\Q(\varepsilon) : \Q] \geq 3$ et que l'équation $g_\varepsilon(X,Y) = k p_1^{Z_1}\cdots p_t^{Z_t}$ ait une solution $(x,y ,z_1,\ldots,z_t)$ dans $\Z^2 \times  \N^t$ vérifiant $xy\not=0$ avec $\pgcd(xy,p_1\cdots p_t)=1$, est fini. }

\bigskip 
Étant donné que nos démonstrations reposent sur le théorème du sous--espace de Schmidt (cf{.}~Théorème $\ref{Theoreme:SousEspace}$),
elles ne permettent pas d'obtenir des énoncés effectifs: elles conduisent à des énoncés quantitatifs avec des bornes pour le nombre de solutions, mais pas à des majorations pour les solutions elles-mêmes. Obtenir des résultats effectifs, au moins pour certains cas particuliers du théorème $\ref{Theoreme:Principal}$, sera l'objet de travaux ultérieurs.


\section{Lien avec un théorème de Vojta}

Le théorème 2.4.1 de Vojta dans \cite{MR883451} (dont la démonstration repose aussi sur le théorème $\ref{Theoreme:SousEspace}$) permet de montrer que l'ensemble des solutions de l'équation $(\ref{Equation:5variables})$ est {\it dégénéré} au sens de \cite{MR883451}, c'est-à-dire contenu dans un sous--ensemble fermé propre pour la topologie de Zariski.
Voici l'énoncé du corollaire 2.4.2 de \cite{MR883451}.

\begin{thm}[Vojta]\label{Theoreme:Vojta}
Soit $D$ un diviseur de $\P^n$ ayant au moins $n+2$ composantes distinctes. Alors tout ensemble de points $D$--entiers de $\P^n$ est dégénéré. \end{thm}

Montrons comment appliquer ce résultat avec $n=4$ à l'équation $(\ref{Equation:5variables})$.
Nous avons vu (au \S$\ref{S:Enonces}$) que le corollaire \ref{thm:5variables} était équivalent au théorème \ref{Theoreme:Principal} . Dans le même esprit, quitte à agrandir $S$, nous pouvons supposer que $\alpha_1=\alpha_2=\alpha_3=\mu=1$. Si $(x,y,\varepsilon_1,\varepsilon_2, \varepsilon)\in\OS^2\times (\OS^\times)^3$ est une solution, alors $x- y$, $x-\varepsilon_1y$ et $x-\varepsilon_2y$ sont des $S$--unités.

Introduisons des coordonnées projectives $(X:Y:Z:E_1:E_2)$ sur $\P^4(K)$ et considérons le diviseur $D$ défini par l'hypersurface 
 $$
  Z\, E_1 \, E_2 \cdot (X- Y) (XZ- E_1 Y) (XZ- E_2 Y) =0
$$
qui a $6$ composantes irréductibles. 
Montrons que chaque solution $\mathbfp=(x,y,\varepsilon_1,\varepsilon_2,\varepsilon)\in\OS^2\times (\OS^\times)^3$ de l'équation 
$(\ref{Equation:5variables})$
en les inconnues $(X,Y,E_1,E_2,E)$ 
définit un point entier sur la variété ouverte $\P^4(K)\setminus D$, à savoir le point $\mathbfptilde=(x:y:1:\varepsilon_1:\varepsilon_2)$. En effet, les coordonnées de $\mathbfptilde$ sont dans $\OS$, et le point $\mathbfptilde$ ne se réduit pas modulo un premier en dehors de $S$ à un point d'un des hyperplans $Z=0$, $E_1=0$, $E_2=0$, $X-Y=0$,
ni à un point d'une des hypersurfaces $XZ-E_1Y=0$,
$XZ-E_2Y=0$,
car les nombres $\varepsilon$, $\varepsilon_1$, $\varepsilon_2$, $x-y$, $x- \varepsilon_1 y$ et $x- \varepsilon_2 y $ sont des $S$--unités. Le théorème $\ref{Theoreme:Vojta}$ montre que les solutions de $(\ref{Equation:5variables})$ ne sont pas Zariski denses dans $V=\P^4(K)\setminus D$. En désignant par $\calE$ l'ensemble des $(x,y,\varepsilon_1,\varepsilon_2)\in\OS^2\times (\OS^\times)^2$ qui vérifient 
$$ 
 (x- y) (x- \varepsilon_1 y) (x- \varepsilon_2 y)\in\OS^\times,
$$
cela signifie qu'il existe un polynôme non nul $P$ en les variables $(X,Y,E_1,E_2)$ qui s'annule au point $(x:y: \varepsilon_1:\varepsilon_2)$ chaque fois que $(x,y,\varepsilon_1,\varepsilon_2)$ est dans $\calE$. On peut remarquer que si 
$(x,y,\varepsilon_1,\varepsilon_2)\in\calE$, alors $(x,y,\varepsilon_2,\varepsilon_1) $ et $(y,x,\varepsilon_1,\varepsilon_2)$ sont aussi dans $\calE$, de même que $(tx,ty,\varepsilon_1,\varepsilon_2)$ pour tout $t$ dans $\OS^\times$. Comme on peut supposer sans perte de généralité que le groupe $\OS^\times$ est infini, on en déduit que le polynôme $P(Tx,Ty, \varepsilon_1, \varepsilon_2)\in K[T]$ est nul. Donc 
$$
P(Tx,Ty,\varepsilon_2,\varepsilon_1)=
P(Ty,Tx,\varepsilon_1,\varepsilon_2)=
P(Tx,Ty, \varepsilon_1, \varepsilon_2)=
P(Ty,Tx,\varepsilon_2,\varepsilon_1)=0.
$$
Il ne semble pas que ces arguments permettent de terminer la démonstration du théorème $\ref{Theoreme:Principal}$; pour y parvenir nous allons revenir à la source, à savoir le théorème du sous--espace de Schmidt. Au préalable, rappelons quelques outils diophantiens.


\section{Outils diophantiens} \label{S:Outils diophantiens} 
 
Soient $K$ un corps de nombres de degré $d=[K:\Q]$ sur $\Q$, $D_K$ son discriminant, $h_K$ son nombre de classes, $R_K$ son régulateur, $M_K$ l'ensemble des places de $K$ et $S_\infty$ l'ensemble des places archimédiennes de $K$. 

Le rang $r$ du groupe des unités de $K$ est $r_1+r_2-1$, où $r_1$ est le nombre de places réelles de $K$ et $r_2$ le nombre de places archimédiennes non réelles. On a $d=r_1+2r_2$ et on note $d_v=1$ pour $v$ une place réelle, $d_v=2$ pour $v$ une place archimédienne non réelle, de sorte que $d=\sum_{v\in S_\infty} d_v$. 
Nous adopterons la normalisation des valeurs absolues 
utilisée dans \cite{MR1373714} et \cite{MR2232500} :
pour $\alpha\in K$, 
$$
|\alpha|_v=\begin{cases}
|\sigma(\alpha)| & \hbox{ si $v$ est une place archimédienne liée à un plongement réel $\sigma:K\rightarrow \bR$,}
\\
|\sigma(\alpha)|^2 & \hbox{ si $v$ est une place archimédienne liée à un plongement non réel $\sigma:K\rightarrow \C$,}
\\
\rmN(\gothP)^{-\ord_\gothP(\alpha) } & \hbox{ si $v$ est une place ultramétrique liée à un idéal premier $\gothP$ de $\OK$.}
\end{cases}
$$
On notera $\rmh(\alpha)$ la hauteur logarithmique absolue d'un élément $\alpha$ de $K$ (c'est la notation de \cite{MR2232500}, alors que ce qui est noté $ \rmh(\alpha)$ dans \cite{MR1373714}, \cite{MR1117339} et \cite{MR1004133} est pour nous $\exp\{ \rmh(\alpha)\}$) :
$$
\rmh(\alpha)=\frac{1}{d}\sum_{v\in M_K} \log \max\{1,|\alpha|_v\}.
$$
On désigne par $\delta_K$ un nombre réel positif tel que tout entier algébrique non nul $\alpha\in K$ qui n'est pas une racine de l'unité vérifie $\rmh(\alpha)\ge \delta_K /d$. Des minorations explicites de $\delta_K$ (dues à Blanksby et Mongomery, Dobrowolski et Voutier) sont données dans \cite{MR1373714}. 

Rappelons (cf{.}~\cite{MR1004133}, \S2 ou bien \cite{MR971998}, \S1) la définition de la   $S$--norme, dénotée $\rmN_S$, associée à un ensemble fini $S$ de places de $K$ contenant les places archimédiennes. Soit $\alpha\in K^\times$. L'idéal principal fractionnaire $(\alpha)$ de $\OK$ engendré par $\alpha$ s'écrit de manière unique comme un produit $\gothA \gothB$, où $\gothA$ est un idéal fractionnaire de $\OK$ qui est produit de puissances d'idéaux premiers qui ne sont pas au--dessus des places finies de $S$, tandis que $\gothB$ est un idéal fractionnaire de $\OK$ qui est produit de puissances d'idéaux premiers au--dessus des places finies de $S$. Alors $N_S(\alpha)$ est la norme de l'idéal $\gothA$. 

Ainsi, par définition, un élément $\alpha$ de $K$ appartient à $\OS$ si et seulement si l'idéal $\gothA$ est un idéal entier de $\OK$. 
De plus, les $S$--unités de $K$ sont les éléments $\varepsilon$ de $\OS$ qui vérifient $\rmN_S(\varepsilon)=1$. Si $S=S_\infty$, alors $\OS=\OK$, $\OS^\times=\OKtimes$ et $\rmN_S$
 coïncide avec $|\rmN_{K/\Q}|$.

Quand $K=\Q$, si $S$ est l'ensemble constitué des places archimédiennes et des places au--dessus des nombres premiers $p_1,\ldots,p_t$, pour un nombre rationnel non nul $\alpha$ écrit sous la forme
$$
\alpha=\pm\left(\prod_{i=1}^t p_i^{a_i}\right)\left(\prod_{j=t+1}^r p_j^{a_j}\right)
$$
avec $a_\ell\in\Z$ $(\ell=1,\dots,r)$ et $p_j\not\in \{p_1,\dots,p_t\}$ ($j=t+1,\dots,r$),
alors on a 
$$
\rmN_S(\alpha)= \prod_{j=t+1}^r p_j^{a_j}. 
$$
On notera $s$ le rang du groupe $\OS^\times$, ce qui fait que le nombre d'éléments\footnote{Noter que dans \cite{MR1373714} le nombre d'éléments de $S$ est noté $s$ (et le rang de $\OS^\times$ est bien sûr $s+1$).} de $S$ est $s-1$. Si $r=r_1+r_2-1$ est le rang du groupe des unités de $K$, alors $t=s-1-r$ est le nombre d'idéaux premiers $\gothP_1,\ldots,\gothP_t$ de $\OK$ au--dessus des places ultramétriques de $K$ dans $S$. On désigne par $\nu$ le plus grand des nombres $|\rmN(\gothP_i)|$ ($1\le i\le t$), avec $\nu=1$ si $t=0$ (c'est-à-dire si $S=S_{\infty}$). 
 La norme indice $S$ d'un élément $\alpha$ est 
$$
\rmN_S(\alpha)=\prod_{v\in S}|\alpha|_v.
$$
Nous utiliserons une base $w_1,\ldots,w_d$ de l'anneau des entiers de $K$ comme $\Z$--module : 
$$
\OK=\Z w_1+\cdots+\Z w_d.
$$
D'après le lemme 5 de \cite{MR1117339}, il existe une telle base pour laquelle le nombre 
$$
\theta=\max\{1\; ,\; \max_{\sigma :K\rightarrow \C } |\sigma(w_1)|+\cdots+|\sigma(w_d)|\}
$$ 
vérifie la majoration $\theta\le |D_K|^{1/2}$.

Notons enfin $d(N)$ la fonction {\it nombre de diviseurs} d'un entier $N>0$ :
$$
d(N):=\sum_{d\mid N} 1.
$$
Nous utiliserons le fait banal que pour tout couple $(n,N)$ d'entiers positifs, le nombre de décomposition de $N$ comme produit de $n$ facteurs $m_1 m_2\cdots m_n=N$ est majoré par $ d(N)^{n-1}$.

L'outil principal des preuves de nos généralisations de ces résultats, que nous utiliserons plusieurs fois par la suite, est une 
conséquence du théorème du sous--espace de Schmidt. L'énoncé précis que nous utiliserons est le sui\-vant, 
dû à Evertse  \cite{MR1359604}, raffiné dans le cas $\ell=2$ par Beukers et Schlickewei \cite{MR1424539}.

\begin{thm}[Evertse, Beukers--Schlickewei]\label{Theoreme:SousEspace}
Soient $K$ un corps de nombres, 
$\delta_1,\ldots,\delta_\ell$ des \'{e}l\'{e}ments non nuls de $K$ 
et $S$ un ensemble fini de places de $K$ de cardinal $s$. 
Alors les solutions $(x_1,\ldots,x_\ell)\in (\OS^{\times})^\ell$ de l'\'{e}quation 
$$
\delta_1X_1+\delta_2X_2+\cdots+\delta_\ell X_\ell=1,
$$
pour lesquelles aucune sous--somme stricte
$$
\sum_{i\in I} \delta_ix_i\qquad (\emptyset\not = I\subset \{1,\ldots,\ell\})
$$
ne s'annule, sont en nombre fini major\'{e} par 
\begin{equation}
\label{equation:sousespace}
 \left\{
\begin{array}{ll}
(2^{33}(\ell +1)^2)^{\ell^3 s}&\mbox{pour }\,  \ell \geq 3,\\
\\
2^{8s+24} &\mbox{pour }\,  \ell =2,\\
\end{array} \right.
\end{equation}
\end{thm}

Notons à ce propos qu'un résultat beaucoup plus général a été obtenu par Amoroso et Viada (théorème 6.2 de \cite{AV}), qui autorisent 
$K$ à être un corps de caractéristique nulle et $G$ un sous--groupe de $K^\times$ de rang fini (cette situation plus générale est aussi celle considérée par Beukers et Schlickewei dans \cite{MR1424539} pour le cas $\ell=2$). Leur énoncé raffine un résultat antérieur de Schmidt, Evertse, van der Poorten et Schlickewei, qui comportait une exponentielle de plus.
La borne qu'ils obtiennent dans ce cadre bien plus général pour le nombre de solutions est un tout petit peu moins précise que celle du théorème $\ref{Theoreme:SousEspace}$, puisqu'ils ont 
$(8\ell)^{4\ell^4(\ell+s+1)}$ à la place de $(\ref{equation:sousespace})$, et c'est pourquoi nous utilisons le résultat plus ancien d'Evertse. 

Nous utiliserons plusieurs fois le théorème $\ref{Theoreme:SousEspace}$, pour $\ell=2$ puis pour $\ell=5$, qui donnent respectivement pour la constante $(\ref{equation:sousespace})$ les valeurs 
$\kappa_3-1$ et $\kappa_4-1$, avec  
$$ 
\kappa_3=1+ 2^{8s+24}
\quad\hbox{et}\quad
\kappa_4=1+ (2^{4375}3^{250})^s.
$$

Le terme dominant de nos estimations provient du lemme suivant. 

\begin{lem}\label{Lemme:ST-A15}
Soient $K$ un corps de nombres et $S$ un ensemble fini de places de $K$ contenant les places archimédiennes. Il existe une constante positive $\kappa_5$, explicitement calculable et ne dépendant que de $K$ et $S$, telle que, pour tout entier $m>0$, il existe un sous--ensemble $A_1(m)$ de $\OS\setminus\{0\}$, ayant au plus $\kappa_5m$ éléments, avec la propriété suivante : pour tout $\beta\in\OS$ vérifiant $\rmN_S (\beta)=m$, il existe $\varepsilon\in\OS^\times$ et $\gamma\in A_1(m)$ satisfaisant $\beta=\varepsilon\gamma.$
\end{lem}
 
 Nous obtiendrons la valeur explicite suivante pour la constante: 
$$
\kappa_5=2^{r +1}\pi^{r_2} |D_K|^{-1/2} e^{c_3 dR_K} \nu^{tdh_K} (1+\theta)^{d}
$$
avec
\begin{equation}\label{equation:c3}
c_3=\frac{1}{2}r^{r+1} \delta_K^{-(r-1)}.
\end{equation}

La première étape de la démonstration du lemme $\ref {Lemme:ST-A15}$ consistera à utiliser le lemme 9 de \cite{MR1117339} que voici.

\begin{lem}\label{Lemme:9EG}
Soit $\beta\in \OS$. Il existe $\eta_1\in\OS^\times$ tel que $\alpha=\eta_1\beta$ soit dans $\OK$ et vérifie
$$
|\rmN_{K/\Q}(\alpha)|\le \nu^{tdh_K}\rmN_S(\beta).
$$
\end{lem}

Les liens avec les notations $|\alpha|_K$ et $|\beta|_S$ de \cite{MR1117339} sont
$$
|\alpha|_K=|\rmN_{K/\Q}(\alpha)|^{1/d} \quad\hbox{ et }\quad |\beta|_S=\rmN_S(\beta)^{1/d}.
$$
Notons que le lemme $\ref{Lemme:9EG}$ est trivial quand $\beta=0$ avec $\eta_1=1$ et $\alpha=0$.
 
 \smallskip
Nous utiliserons ensuite le résultat suivant qui est un scolie découlant du lemme 2 de \cite{MR1373714}, avec la constante $c_3$ introduite en $(\ref{equation:c3})$.
 
\begin{lem}\label{Lemme:ScolieBG}
Soit $\alpha\in \OK\setminus\{0\}$. Notons $M=|\rmN_{K/\Q} (\alpha)|$. Alors il existe $\eta_2\in\OKtimes$ tel que, pour tout $v\in S_\infty$, on ait 
$$
M^{d_v/d} e^{-c_3 R_K}\le|\alpha\eta_2|_v\le M^{d_v/d} e^{c_3 R_K}.
$$
\end{lem}

Une variante du lemme $\ref{Lemme:ScolieBG}$ s'obtient comme scolie du lemme 3 de \cite{MR2232500}, qui fournit le même résultat, mais avec la constante $c_3$ remplacée par 
$$
c'_3=
\begin{cases}
0 & \hbox{ si $r=0$,}\\
1/d & \hbox{ si $r=1$,}\\
29 e (r!) r \sqrt{r-1} \log d & \hbox{ si $r\ge 2$.}\\
\end{cases}
$$
Le lemme 1 de \cite{MR1004133} et le lemme A15 de \cite{MR88h:11002} donnent le même énoncé mais n'explicitent pas $c_3$. 

\begin{lem}\label{Lemme:comptage}
Soit $Q$ un nombre réel $\ge 1$. Alors l'ensemble 
$$
\Gamma:=\bigl\{
\gamma\in\OK\; \mid  \; 
|\gamma|_v\le Q^{d_v} \quad\hbox{pour tout $v\in S_{\infty}$}
\bigr\}
$$
est fini et a au plus 
$$
2^{r+1}\pi^{r_2} (Q+\theta)^d |D_K|^{-1/2} 
$$
éléments. 

\end{lem}

\begin{proof}[{Démonstration}]
Choisissons un ordre sur $S_\infty$ (en numérotant en premier les places réelles et en regroupant une place complexe et sa conjuguée) et notons $\usigma$ le plongement canonique de $K$ dans $\bR^{r_1}\times\C^{r_2}$ :
$$
\begin{matrix}
\usigma : & K &\longrightarrow &\bR^{r_1}\times\C^{r_2}\qquad
\\
&\alpha&\longmapsto & \bigl(\sigma_i(\alpha)\bigr)_{1\le i\le r_1+r_2}\\
\end{matrix}
$$
où $\sigma_1,\ldots,\sigma_d$ sont les plongements de $K$ dans $\C$ avec $\sigma_{r_1+r_2+j}=\overline{\sigma}_{r_1+j}$ pour $1\le j\le r_2$ ($\overline{\sigma}$ désigne le conjugué complexe du plongement $\sigma$). 
Quand on identifie $\bR^{r_1}\times\C^{r_2}$ à $\bR^d$, l'image de $\OK$ est un réseau de $\bR^d$ dont une base est $\usigma(w_1),\ldots,\usigma(w_d)$. Le parallélotope $\calP$ défini par cette base est un domaine fondamental pour le réseau dont le volume est $2^{-r_2}\sqrt{|D_K|}$. 
La réunion des $\gamma+\calP$ quand $\gamma$ décrit $\Gamma$ est une réunion disjointe 
contenue dans le domaine borné
$$
\calB=
\bigl\{
(x_v)_{v\in S_\infty }\in \bR^{r_1}\times\C^{r_2}\;  \mid  \; |x_v|\le Q+\theta \; \; (v\in S_\infty)
\bigr\}.
$$
Le volume de $\calB$ est $2^{r_1}\pi^{r_2} (Q+\theta)^d$. Le lemme $\ref{Lemme:comptage}$ en résulte.

\end{proof}

\begin{proof}[{Preuve du lemme $\ref {Lemme:ST-A15}$}]
Soit $\beta$ un élément de $\OS$ vérifiant $\rmN_S(\beta)\le m$.  
On utilise d'abord le lemme $\ref{Lemme:9EG}$ pour écrire 
 $\beta=\alpha\eta_1^{-1}$ avec $\alpha\in \OK$, $ \eta_1\in\OS^\times$ et 
$$
|\rmN_{K/\Q}(\alpha)|\le \nu^{tdh_K}\rmN_S(\beta).
$$
On utilise ensuite le lemme $\ref{Lemme:ScolieBG}$ 
pour écrire $\alpha=\eta_2^{-1}\gamma$ avec 
$\eta_2\in\OKtimes$ et 
$$
M^{d_v/d} e^{-c_3 R_K}\le|\alpha\eta_2|_v\le M^{d_v/d} e^{c_3 R_K}
\quad\hbox{ pour tout $v\in S_\infty$},
$$ 
avec $M=|\rmN_{K/\Q} (\alpha)|\le \nu^{tdh_K} m$. On utilise enfin le lemme $\ref{Lemme:comptage}$
avec $Q= M^{1/d} e^{c_3 R_K}$ pour dire que l'ensemble 
$$
\Gamma :=\bigl\{
\gamma\in\OK\; \mid  \; 
|\gamma|_v\le Q^{d_v} \quad\hbox{pour tout $v\in S_{\infty}$}
\bigr\}
$$
est fini et a au plus 
$$
 2^{r+1}\pi^{r_2} (Q+\theta)^d |D_K|^{-1/2} 
$$
éléments. On majore $Q+\theta$ par $Q(1+\theta)$: on a
$$
 2^{r+1}\pi^{r_2} (Q+\theta)^d |D_K|^{-1/2} \le \kappa_5 m. 
$$
 Enfin, en posant $\varepsilon=\eta_1^{-1} \eta_2^{-1}$, on a $\beta=\varepsilon\gamma$, ce qui complète la démonstration du lemme $\ref {Lemme:ST-A15}$. 
\end{proof}

\section{Preuve du théorème $\ref{Theoreme:Principal}$} \label{S:PreuveTheoreme}

Nous décomposerons la démonstration en plusieurs étapes. Dans la première, nous mettons en évidence des ensembles finis, indépendants d'une solution éventuelle $(x,y,z,\varepsilon,\varepsilon_1 ,\varepsilon_2,\varepsilon_3)\in\OS^3\times(\OS^\times)^{4}$
 de l'équation diophantienne $(\ref{Equation:ThueGeneralisee})$.

 \subsection{Mise en évidence de sous--ensembles finis} \label{SS:miseenevidence}
Soit $K$ un corps de nombres de degré $d$ sur $\Q$ et soient $\mu, \alpha_1,\alpha_2,\alpha_3$ des éléments non nuls de $K$. Soit $\denominateur$ un entier rationnel positif tel que $\denominateur\alpha_i\in\OS$ pour $i=1,2,3$. Un bon choix pour $q$ est le plus petit commun multiple  (ou ppcm) des coefficients directeurs des polynômes irréductibles sur $\Z$ des nombres algébriques $\alpha_1,\alpha_2,\alpha_3$. 
On a $\rmN_S(\denominateur)\le \denominateur^d$. 
L'existence d'une solution de $(\ref{Equation:ThueGeneralisee})$ implique $\denominateur^3\mu \in\OS$. 
On pose 
$$
k=\rmN_S (\mu)\quad\hbox{ et }\quad m=\rmN_S(\denominateur)^3 k, 
$$
de sorte que $m=\rmN_S(\denominateur^3\mu)$ est un entier rationnel strictement positif. 

L'ensemble $A_2$ formé des triplets $\uk'=(k'_1,k'_2 ,k'_3)$ d'entiers positifs satisfaisant $k'_1k'_2 k'_3=m$ est fini et a au plus 
$d(m)^2$ éléments.
Soit $\uk' \in A_2$ et soit $i\in\{1,2,3\}$. Le lemme $\ref{Lemme:ST-A15}$ montre qu'il existe un sous--ensemble fini $A_1 (k'_i)$ de $\OS\setminus\{0\}$, ayant au plus $\kappa_5 k'_i$ éléments, avec la propriété suivante : pour tout $ \beta\in\OS$ vérifiant $\rmN_S ( \beta)=k'_i$, il existe $\gamma_i\in A_1(k'_i)$ et $w_i\in\OS^\times$ tel que $ \beta=\gamma_i w_i$.

Pour $\uk'=(k'_1,k'_2,k'_3)\in A_2$, on note $A_1 (\uk')$ le sous--ensemble suivant de $\OS^3$ : 
$$
A_1 (\uk')=
 A_1(k'_1)\times A_1(k'_2)\times A_1(k'_3).
$$
Le nombre d'éléments de la réunion des $A_1 (\uk')$ quand $\uk'$ décrit $ A_2$ est majoré par 
$$
\sum_{(k'_1,k'_2,k'_3)\in A_2} 
\kappa_5^3 k'_1 k'_2 k'_3
\le
\kappa_6
$$
avec $$
\kappa_6=d(m)^2
\kappa_5^3 m.
$$
Soient $\ell$ un entier $\ge 2$ et $\delta_1,\ldots,\delta_\ell$ des éléments de $K^\times$. Notons $\udelta=(\delta_1,\ldots,\delta_\ell)\in (K^\times)^\ell$. Le théorème $\ref{Theoreme:SousEspace}$ montre qu'il existe un sous--ensemble fini $A_3^{(\ell)} (\udelta)$ de $\OS^\times$, contenant $1$, ayant un nombre 
d'éléments majoré par $(\ref{equation:sousespace})$, tel que, pour toute solution $(x_1,\ldots,x_\ell)\in (\OS^\times)^\ell$ de l'équation 
$$
\delta_1X_1+ \delta_2X_2+\cdots+\delta_\ell X_\ell=1
$$
pour laquelle aucune sous--somme stricte
$$
\sum_{i\in I} \delta_ix_i
$$ 
ne s'annule, on ait $x_1\in A_3^{(\ell)}(\udelta)$. 

Dans la même veine, nous noterons encore $\widetilde{A}_3^{(\ell)}(\udelta)$ l'ensemble formé par les éléments $x$ de $\OS^\times$ tels que l'un au moins des deux éléments $x$, $x^{-1}$, appartienne à la réunion des ensembles $ A_3^{(\ell)}(\udelta^\sigma)$ avec $\udelta^\sigma=(\delta_{\sigma(1)},\ldots,\delta_{\sigma(\ell)})$, quand $\sigma $ décrit les transpositions $(1\; i)$ de $\{1,\ldots,\ell\}$. 
 
Soit $\uu=(u_1,u_2,u_3)\in (\OS^\times)^3$ tel que $\Card\{u_1\alpha_1,u_2\alpha_2,u_3\alpha_3\} = 3$.
 Pour $i=1,2,3$, posons $\alpha'_i=\alpha_i u_i$,
$$ 
\delta_1=
\frac{\alpha'_3-\alpha'_1}{\alpha'_3-\alpha'_1},
\quad 
\delta_2=\frac{\alpha'_1-\alpha'_1}{\alpha'_3-\alpha'_1} ,\quad \udelta=(\delta_1,\delta_2)
$$ 
et 
$A_4(\uu)
= A_3^{(2)}(\udelta)$.
Alors $ A_4(\uu)$ est un sous--ensemble fini de $\OS^\times$, ayant au plus $\kappa_3$ éléments,
qui contient $1$ et les deux quotients $t_2/t_1$ et $t_3/t_1$ quand $(t_1,t_2,t_n)$ est un triplet d'éléments de $\OS^\times$ satisfaisant
\begin{equation}\label{EquationAlphaPrime}
t_1(\alpha'_2-\alpha'_3)+
t_2(\alpha'_3-\alpha'_1)+
t_3(\alpha'_1-\alpha'_2)=0.
\end{equation}

 \medskip
 \subsection{Utilisation des ensembles finis mis en évidence } 
 Après avoir mis en évidence au lemme $\ref {Lemme:ST-A15}$ et à la section $\ref{SS:miseenevidence}$ 
 les sous--ensembles finis $A_1 (m)$, $ A_2$, $A_1 (\uk')$, $ A_3^{(\ell)}(\udelta)$, $\widetilde{A}_3^{(\ell)}(\udelta)$, $ A_4(\uu)$, nous pouvons procéder à la démonstration du Théorème $\ref{Theoreme:Principal}$ proprement dite. 
 
 Partons d'une solution $(x,y,z,\varepsilon,\varepsilon_1, \varepsilon_2,\varepsilon_3)\in\OS^3\times(\OS^\times)^{4}$
 de l'équation diophantienne $(\ref{Equation:ThueGeneralisee})$ avec $xy\not=0$ et prenons pour $\denominateur$ le plus petit entier rationnel positif tel que $\denominateur\alpha_i\in\OS$ pour $i=1,2,3$. Posons, pour $i=1,2,3$, 
$$
\alphatilde_i=\alpha_i\varepsilon_i, \quad
 \beta_i=x- \alphatilde_i y, \quad
 \beta'_i= \denominateur x-\denominateur \alphatilde_i y = \denominateur \beta_i
\quad\hbox{et}\quad
k'_i=\rmN_S(\beta'_i),
$$
de sorte que $\beta'_1\beta'_2 \beta'_3=\denominateur^3 \mu\varepsilon$ et $k'_1 k'_2 k'_3=m$, où $m=\rmN_S(\denominateur^3\mu)$.
Comme nous l'avons déjà remarqué, l'existence d'une solution implique $\denominateur^3\mu \in\OS$. Par définition de $ A_2$, on a $(k'_1,k'_2,k'_3)\in A_2$. 

La construction de $A_1 (\uk')$ montre qu'il existe $\ugamma=(\gamma_1,\gamma_2,\gamma_3)\in A_1 (\uk')$ et $(w_1,w_2,w_3)\in (\OS^\times)^3$ tels que 
$$
\beta'_i\gamma_i^{-1} = w_i\quad \hbox{pour}\quad i=1,2,3.
$$
Rappelons maintenant que $\Card\{\alphatilde_1 , \alphatilde_2,\alphatilde_3 \}= 3$. Nous allons construire deux sous--ensembles finis $ A_3(\ugamma,i)$ ($i=2,3$) de $\OS^\times$, contenant $1$ et ayant au plus $\kappa_4$ éléments,
tels que $\varepsilon_i/\varepsilon_1 \in  A_3(\ugamma,i) $ pour $i=2,3$. 

Soit $i\in\{2,3\}$. 
On a $\alphatilde_i\not= \alphatilde_1$. Soit $j$ tel que $\{i,j\}=\{2,3\}$. Alors 
 $\alphatilde_1,\alphatilde_i,\alphatilde_j$ sont deux à deux distincts. 
En éliminant $x$ et $y$ entre les trois équations 
$$
x-\alphatilde_1y=\beta_1,\quad
x-\alphatilde_iy=\beta_i,\quad
x-\alphatilde_jy=\beta_j,
$$
c'est-à-dire en écrivant
$$
(x-\alphatilde_1y)(\alphatilde_i-\alphatilde_j)+
(x-\alphatilde_iy)(\alphatilde_j-\alphatilde_1)+
(x-\alphatilde_jy)(\alphatilde_1-\alphatilde_i)
=0, 
$$
on trouve que le nombre 
\begin{equation}\label{Equation:SommeSixTermes}
\SSomme=\beta_1\alphatilde_i-\beta_1\alphatilde_j+\beta_i\alphatilde_j-\beta_i\alphatilde_1+\beta_j\alphatilde_1-\beta_j\alphatilde_i 
\end{equation}
est nul. 
 
 Pour toute permutation 
$(j_1,j_2,j_3)=\bigl(\sigma(1),\sigma(i), \sigma(j)\bigr)$ de $(1,i,j)$, nous pouvons écrire l'égalité $\SSomme=0$, avec $\SSomme$ défini par $(\ref{Equation:SommeSixTermes})$, sous la forme 
\begin{equation}\label{Equation:CinqTermes}
\sgn(\sigma)
\left( \frac{
\alphatilde_{j_3}
}{
\alphatilde_{j_2}
}
- \frac{
\beta_{j_2}\alphatilde_{j_3}
}{
\beta_{j_1}\alphatilde_{j_2}
}
+\frac{
\beta_{j_2}\alphatilde_{j_1}
}{
\beta_{j_1}\alphatilde_{j_2}
}
-
 \frac{
\beta_{j_3}\alphatilde_{j_1} 
}{
\beta_{j_1}\alphatilde_{j_2}
}
+ \frac{
\beta_{j_3}
}{
\beta_{j_1} 
}
\right)
=1,
\end{equation}
 où $\sgn(\sigma)$ est la signature de la permutation $\sigma$.

\subsubsection{Premier cas : Aucune sous--somme stricte de $\SSomme$ ne s'annule}\label{SS:SousSommesNonNulles} 
Dans ce cas, on prend $\ell=5$, on pose
 $$ 
 \delta_1=
 \frac{
\alpha_i }{
\alpha_1 
},
\quad
\delta_2=- \frac{
\gamma_1 \alpha_i 
}{
\gamma_j \alpha_1 },
\quad
\delta_3=\frac{
\gamma_1 \alpha_j 
}{
\gamma_j \alpha_1 
},
\quad
\delta_4=
-
 \frac{
\gamma_i\alpha_j 
}{
\gamma_j \alpha_1 
},
\quad
\delta_5=
 \frac{ 
\gamma_i 
}{
\gamma_j },
$$
et il s'avère que l'on peut prendre l'ensemble $ A_3(\ugamma,i)$ égal à 
$ A_3^{(5)}(\udelta)$ avec $\underline\delta = (\delta_1, \delta_2, \dots , \delta_5)$. En effet, cet ensemble a au plus $\kappa_4$
éléments, d'après le théorème $\ref{Theoreme:SousEspace}$. 
Pour montrer que le quotient $\varepsilon_i/\varepsilon_1$ appartient à $ A_3^{(5)}(\udelta)$, on écrit 
 $(\ref{Equation:CinqTermes})$ avec $(j_1,j_2,j_3)=(j,1,i)$ sous la forme 
 $$ 
 \delta_1 
 \frac{
\varepsilon_i }{
\varepsilon_1 
}
+
\delta_2 \frac{
w_1 
}{
w_j }
 \frac{
 \varepsilon_i 
}{
 \varepsilon_1 }
+
\delta_3
\frac{
w_1 
}{
w_j 
}\frac{
 \varepsilon_j 
}{
 \varepsilon_1 
}
+
\delta_4
 \frac{
w_i
}{
w_j 
}
\frac{
 \varepsilon_j 
}{
 \varepsilon_1 
}
+
\delta_5 \frac{ 
w_i 
}{
w_j } = 1.
$$
et on utilise la définition de $ A_3^{(5)}(\udelta)$.

\subsubsection{Deuxième cas: Au moins une sous--somme stricte s'annule}\label{SSS:deuxiemecas}

\medskip
\noindent
{\bf \ref{SSS:deuxiemecas}.1 Preuve que la somme $\SSomme$ ne se décompose pas en trois sous--sommes nulles de deux termes chacune}

Pour chacune des $15$ partitions de l'ensemble 
$$
\calE=\bigl\{
(1,i), \;
(1,j), \;
(i,1), \;
(i,j), \;
(j,1), \;
(j,i)
\bigr\}
$$ 
en trois sous--ensembles de deux éléments, disons $\calE=\calE_1\cup\calE_2\cup\calE_3$, au moins un des ensembles $\calE_1$, $\calE_2$, $\calE_3$ peut s'écrire $\{(j_1,j_2)\; , \; (\ell_1,\ell_2)\}$ avec 
$$
\hbox{ $j_1=\ell_1$ ou $j_2=\ell_2$ ou $(j_1,j_2)=(\ell_2,\ell_1)$.}
$$ 
Aucune sous--somme $\pm(\beta_{j_1}\alphatilde_{j_2}-\beta_{\ell_1}\alphatilde_{\ell_2})$ sur un ensemble à deux termes $\{(j_1,j_2)\; , \; (\ell_1,\ell_2)\}$ avec $j_1=\ell_1$ (donc $j_2\not=\ell_2$) ne s'annule, car 
$$
\beta_{j_1}(\alphatilde_{j_2}-\alphatilde_{\ell_2})\not=0.
$$
De même, aucune sous--somme sur un ensemble $\{(j_1,j_2)\; , \; (\ell_1,\ell_2)\}$ avec $j_2=\ell_2$ (donc $j_1\not=\ell_1$) ne s'annule, 
car
$$
(\beta_{j_1}-\beta_{\ell_1}) \alphatilde_{j_2} \not=0.
$$
 Enfin, aucune sous--somme sur un ensemble $\{(j_1,j_2)\; , \; (\ell_1,\ell_2)\}$ avec $(j_1,j_2)=(\ell_2,\ell_1)$ ne s'annule, car 
$$
\beta_{j_1}\alphatilde_{j_2}-\beta_{j_2}\alphatilde_{j_1}
=
(\alphatilde_{j_2} - \alphatilde_{j_1})x
\not=0.
$$
On en déduit que la somme $\SSomme$ de $(\ref{Equation:SommeSixTermes})$ ne se décompose pas en trois sous--sommes nulles de deux termes. 

\medskip
\noindent
{\bf \ref{SSS:deuxiemecas}.2  Cas où $\SSomme$ se décompose en deux sous--sommes nulles de trois termes}

Supposons $\SSomme=\SSomme_1+\SSomme_2$ avec $\SSomme_1$ et $\SSomme_2$ sous--sommes de trois termes et $\SSomme_1=\SSomme_2=0$. Alors aucune sous--somme stricte de $\SSomme_1$ ni de $\SSomme_2$ ne s'annule.

\noindent
$\bullet$ 
{\bf Supposons que deux des trois indices des $\beta$ dans $\SSomme_1$ sont égaux.}
Il existe alors une permutation 
$(j_1,j_2,j_3)=\bigl(\sigma(1),\sigma(i), \sigma(j)\bigr)$ telle que l'on ait, soit 
$$
\sgn(\sigma) \SSomme_1
=\beta_{j_1}\alphatilde_{j_2}-\beta_{j_1}\alphatilde_{j_3} + \beta_{j_3}\alphatilde_{j_1} 
\quad\hbox{et}\quad 
\sgn(\sigma) \SSomme_2= \beta_{j_2}\alphatilde_{j_3}-\beta_{j_2}\alphatilde_{j_1} - \beta_{j_3}\alphatilde_{j_2},
$$
soit 
$$
\sgn(\sigma) \SSomme_1
=\beta_{j_1}\alphatilde_{j_2}-\beta_{j_1}\alphatilde_{j_3} - \beta_{j_3}\alphatilde_{j_2} 
\quad\hbox{et}\quad 
\sgn(\sigma) \SSomme_2= \beta_{j_2}\alphatilde_{j_3}-\beta_{j_2}\alphatilde_{j_1} + \beta_{j_3}\alphatilde_{j_1}.
$$
Commençons par le premier cas. Posons $\udelta=(\delta_1,\delta_2)$ et $\udelta'=(\delta'_1,\delta'_2)$ avec 
$$ 
\delta_1=\frac{\alpha_{j_2}}{\alpha_{j_3}},
\quad
\delta_2=\frac{\gamma_{j_3}\alpha_{j_1}}{\gamma_{j_1}\alpha_{j_3}},
\qquad 
\delta'_1=\frac{\alpha_{j_1}}{\alpha_{j_3}},
\quad
\delta'_2=\frac{\gamma_{j_3}\alpha_{j_2}}{\gamma_{j_2}\alpha_{j_3}},
$$
de sorte que
$$
\delta_1\frac{\varepsilon_{j_2}}{\varepsilon_{j_3}} +
\delta_2
\frac{w_{j_3} }{w_{j_1} }
\frac{ \varepsilon_{j_1}}{ \varepsilon_{j_3}}= 
\delta'_1\frac{\varepsilon_{j_1}}{\varepsilon_{j_3}}
+
\delta'_2\frac{w_{j_3} }{w_{j_2} }
\frac{ \varepsilon_{j_2}}{ \varepsilon_{j_3}}
=1.
$$
Il en résulte que si on définit $\xi_1 := \varepsilon_{j_2}/\varepsilon_{j_3}$ et $\xi_2 := \varepsilon_{j_1}/\varepsilon_{j_3}$, 
on a 
$\xi_1 \in  A_3^{(2)}(\udelta)$ et $\xi_2 \in  A_3^{(2)}(\udelta')$,
tandis que 
$\xi_1^{-1}=\varepsilon_{j_3}/\varepsilon_{j_2}\in\widetilde{A}_3^{(2)}(\udelta)$ et 
${\xi_2}^{-1}= \varepsilon_{j_3}/\varepsilon_{j_1} \in\widetilde{A}_3^{(2)}(\udelta')$. Comme 
$$
\xi_1 \xi_2^{-1}= \frac{\varepsilon_{j_2}}{\varepsilon_{j_1} }\quad\hbox{et}\quad
\xi_1^{-1}\xi_2= \frac{\varepsilon_{j_1}}{\varepsilon_{j_2}}, 
$$
l'un des nombres $\xi_1 ,\; \xi_2 ,\; \xi_1^{-1} ,\; {\xi_2}^{-1} ,\; 
\xi_1 {\xi_2}^{-1} ,\; \xi_1^{-1}\xi_2$ est égal à $\varepsilon_i/\varepsilon_1$. 
Enfin, comme $1\in  A_3^{(2)}(\udelta)\cap  A_3^{(2)}(\udelta')$, 
on pose
$$
 A_3(\ugamma,i)=
\bigl\{
\vartheta_1 \vartheta_2\in \OS^\times\; \mid  \;
\vartheta_1 \in\widetilde{A}_3^{(2)}(\udelta),\; 
\vartheta_2 \in\widetilde{A}_3^{(2)}(\udelta')
\bigr\}.
$$
On vérifie encore que $ A_3 (\ugamma,i)$ a au plus $\kappa_4$
éléments et qu'il contient $\varepsilon_i/\varepsilon_1$.

Passons au deuxième cas. Posons
$$ 
\delta_1=\frac{ \alpha_{j_3}}{ \alpha_{j_2}},
\quad
\delta_2
=\frac{\gamma_{j_3}}{\gamma_{j_1}},
\qquad
\delta'_1=\frac{\alpha_{j_3}}{ \alpha_{j_1}},
\quad
\delta'_2=\frac{\gamma_{j_3}}{\gamma_{j_2}}\cdotp
$$
Nous avons
$$
\delta_1
\frac{ \varepsilon_{j_3}}{ \varepsilon_{j_2}}+
\delta_2 \frac{w_{j_3}}{w_{j_1}} =
\delta'_1\frac{\varepsilon_{j_3}}{ \varepsilon_{j_1}} +
\delta'_2\frac{w_{j_3}}{w_{j_2}}=1.
$$
Par conséquent, si nous posons $\udelta=(\delta_1,\delta_2)$, $\udelta'=(\delta'_1,\delta'_2)$, $\xi_1 :=\varepsilon_{j_3}/ \varepsilon_{j_2}$
 et 
$\xi_2 :=\varepsilon_{j_3}/ \varepsilon_{j_1}$,
nous avons $\xi_1 \in  A_3^{(2)}(\udelta)$, $\xi_2 \in  A_3^{(2)}(\udelta')$, et aussi 
$$
\xi_1^{-1}=\frac{\varepsilon_{j_2}}{\varepsilon_{j_3}}
\in\widetilde{A}_3^{(2)}(\udelta), 
\quad
\xi_2^{-1} :=\frac{\varepsilon_{j_1}}{ \varepsilon_{j_3}}\in\widetilde{A}_3^{(2)}(\udelta').
$$
Enfin 
$$
\xi_1^{-1}\xi_2=\frac{\varepsilon_{j_2}}{\varepsilon_{j_1}} \quad\hbox{et}\quad 
\xi_1 \xi_2^{-1}=\frac{\varepsilon_{j_1}}{ \varepsilon_{j_2}}\cdotp 
$$ 
Nous posons donc
$$
 A_3(\ugamma,i)=
\bigl\{
\vartheta_1 \vartheta_2\; \mid  \;
\vartheta_1 \in\widetilde{A}_3^{(2)}(\udelta)\; 
,\; 
\vartheta_2 \in\widetilde{A}_3^{(2)}(\udelta')
\bigr\}.
$$

 \bigskip

\noindent
$\bullet$ 
{\bf Supposons que les trois indices des $\beta$ dans $\SSomme_1$ sont distincts mais que deux des trois indices des $\alphatilde$ dans $\SSomme_1$ sont égaux. }
Il existe alors une permutation $(j_1,j_2,j_3)=\bigl(\sigma(1),\sigma(i), \sigma(j)\bigr)$ telle que 
$$
\sgn(\sigma) \SSomme_1
= \beta_{j_1}\alphatilde_{j_2}-\beta_{j_2}\alphatilde_{j_1} - \beta_{j_3}\alphatilde_{j_2}
\quad\hbox{et}\quad
\sgn(\sigma) \SSomme_2=-\beta_{j_1}\alphatilde_{j_3}+\beta_{j_2}\alphatilde_{j_3} + \beta_{j_3}\alphatilde_{j_1}.
$$
Dans ce cas on pose
$$ 
\delta_1=\frac{\gamma_{j_3} }{\gamma_{j_1} },
\quad
\delta_2
=\frac{\gamma_{j_2}\alpha_{j_1}}{\gamma_{j_1}\alpha_{j_2}},
\qquad 
\delta'_1=\frac{\gamma_{j_2} }{\gamma_{j_1} },
\quad
\delta'_2=\frac{\gamma_{j_3}\alpha_{j_1}}{\gamma_{j_1}\alpha_{j_3}},
$$
de sorte que
$$ 
\delta_1 \frac{w_{j_3} }{w_{j_1} } 
+
\delta_2 
\frac{w_{j_2} }{w_{j_1} }
\frac{\varepsilon_{j_1}}{\varepsilon_{j_2}}
=
\delta'_1\frac{w_{j_2} }{w_{j_1} }
+
\delta'_2 
\frac{w_{j_3}}{ w_{j_1} }
\frac{ \varepsilon_{j_1}}{ \varepsilon_{j_3}}
=1.
$$
Ainsi les $S$--unités
$$
\xi_1 :=\frac{w_{j_2}}{ w_{j_1}},\quad
\xi_2 :=\frac{w_{j_3} }{ w_{j_1}}, \quad
\xi_3 :=
\frac{w_{j_2} }
{w_{j_1} }
\frac{\varepsilon_{j_1}}
{ \varepsilon_{j_2}}
\quad\hbox{et}\quad
\xi_4 :=
\frac
{w_{j_3} }
{w_{j_1} }
\frac
{ \varepsilon_{j_1}}
{ \varepsilon_{j_3}}
$$
vérifient
$$
\xi_1 \in  A_3^{(2)}(\udelta'),
\quad
\xi_2 \in  A_3^{(2)}(\udelta),
\quad
\xi_3 \in\widetilde{A}_3^{(2)}(\udelta),
\quad
\xi_4 \in\widetilde{A}_3^{(2)}(\udelta'),
$$
avec $\udelta=(\delta_1,\delta_2)$, $\udelta'=(\delta'_1,\delta'_2)$, 
et on a 
\begin{align}\notag
&\xi_3\xi_1^{-1}=\frac{\varepsilon_{j_1}}{\varepsilon_{j_2}},\quad
\xi_3^{-1}\xi_1=\frac{\varepsilon_{j_2}}{\varepsilon_{j_1}},\quad
\xi_4\xi_2^{-1}=\frac{\varepsilon_{j_1}}{\varepsilon_{j_3}},\\
\notag
&\xi_2\xi_4^{-1}=\frac{\varepsilon_{j_3}}{\varepsilon_{j_1}},\quad 
 \xi_2^{-1}\xi_3^{-1} \xi_1\xi_4=\frac{\varepsilon_{j_2}}{\varepsilon_{j_3}},\quad 
 \xi_2\xi_3 \xi_1^{-1} \xi_4^{-1}=\frac{\varepsilon_{j_3}}{\varepsilon_{j_2}}\cdotp
\end{align}
 On pose donc
$$
 A_3(\ugamma,i)=
\bigl\{
\vartheta_1 \vartheta_2 \vartheta_3 \vartheta_4\; \mid \; 
\vartheta_1, \vartheta_2\in 
\widetilde{A}_3^{(2)}(\udelta)
\; , \; 
\vartheta_3, \vartheta_4\in 
\widetilde{A}_3^{(2)}(\udelta')
\bigr\}.
$$ 

 \bigskip

\noindent
$\bullet$ 
{\bf Supposons que les trois indices des $\beta$ ainsi que les trois indices des $\alphatilde$ dans $\SSomme_1$ sont distincts. }
Il existe alors une permutation $(j_1,j_2,j_3)=\bigl(\sigma(1),\sigma(i), \sigma(j)\bigr)$ telle que 
$$
\sgn(\sigma) \SSomme_1
=\beta_{j_1}\alphatilde_{j_2}+\beta_{j_2}\alphatilde_{j_3} + \beta_{j_3}\alphatilde_{j_1} 
\quad\hbox{et}\quad
\sgn(\sigma) \SSomme_2= - \beta_{j_1}\alphatilde_{j_3} - \beta_{j_2}\alphatilde_{j_1} - \beta_{j_3}\alphatilde_{j_2}.
$$
On pose
$$ 
\delta_1
=-\frac{\gamma_{j_3}\alpha_{j_1}}{\gamma_{j_1}\alpha_{j_2}},
\quad
\delta_2=-\frac{\gamma_{j_2} \alpha_{j_3} }{\gamma_{j_1} \alpha_{j_2}},
\qquad 
\delta'_1=-\frac{ \gamma_{j_2} \alpha_{j_1}}{\gamma_{j_1} \alpha_{j_3}},
\quad
\delta'_2=-\frac{\gamma_{j_3} \alpha_{j_2}}{\gamma_{j_1} \alpha_{j_3}},
$$
de sorte que 
$$ 
\delta_1
\frac{w_{j_3}}{w_{j_1} }
\frac{ \varepsilon_{j_1}}{ \varepsilon_{j_2}}
+ 
\delta_2\frac{w_{j_2} }{w_{j_1} }
\frac{ \varepsilon_{j_3} }{ \varepsilon_{j_2}}
= 
\delta'_1
\frac{ w_{j_2} }{w_{j_1} }
\frac{ \varepsilon_{j_1}}{ \varepsilon_{j_3}}
+
\delta'_2
\frac{w_{j_3} } {w_{j_1} }
\frac{ \varepsilon_{j_2}}{ \varepsilon_{j_3}}
=1. 
$$ 
Ainsi les $S$--unités
$$
\xi_1 :=
\frac{ w_{j_3} }{w_{j_1} }
\frac{ \varepsilon_{j_1}}{ \varepsilon_{j_2}}
,\quad
\xi_2 := \frac{w_{j_2} }{w_{j_1} } 
\frac{ \varepsilon_{j_3}}{ \varepsilon_{j_2}} 
, \quad
\xi_3 := \frac{w_{j_2} }{ w_{j_1} }
\frac{ \varepsilon_{j_1}}{ \varepsilon_{j_3}}
\quad\hbox{et}\quad
\xi_4 := \frac{w_{j_3} }{ w_{j_1} }
 \frac{ \varepsilon_{j_2} }{ \varepsilon_{j_3}}
$$
vérifient
$$
\xi_1,\xi_2 \in\widetilde{A}_3^{(2)}(\udelta),
\quad
\xi_3,\xi_4 \in\widetilde{A}_3^{(2)}(\udelta'),
$$
avec $\udelta=(\delta_1,\delta_2)$, $\udelta'=(\delta'_1,\delta'_2)$, et on a 
$$
\frac{\xi_3\xi_4}{\xi_1\xi_2}=\frac{\varepsilon_{j_2}^3}{\varepsilon_{j_3}^3},
\quad
\frac{\xi_1^2\xi_3}{\xi_2\xi_4^2}=\frac{\varepsilon_{j_1}^3}{\varepsilon_{j_2}^3},
\quad
\frac{\xi_2^2\xi_4}{\xi_1\xi_3^2}=
\frac{\varepsilon_{j_3}^3}{\varepsilon_{j_1}^3},
$$
de sorte que l'on a aussi en mains les inverses de ces trois derniers nombres. On pose donc
$$
 A_3(\ugamma,i)=
\left\{
\zeta \in \OS^\times\; \mid \; \hbox{il existe $\vartheta_1, \vartheta_2, \in\widetilde{A}_3^{(2)}(\udelta)$ et 
$\vartheta_3, \vartheta_4 \in\widetilde{A}_3^{(2)}(\udelta') $ 
tels que $\zeta^3\in
\left\{
\frac{\displaystyle \vartheta_1^2 \vartheta_3}{ \displaystyle \vartheta_2 \vartheta_4^2}
\; , \; 
\frac{\displaystyle \vartheta_3 \vartheta_4}{ \displaystyle \vartheta_1 \vartheta_2}
\right\}
$}
\right\}. 
$$

 \bigskip

\noindent
{\bf \ref{SSS:deuxiemecas}.3 Cas où $\SSomme$ se décompose en exactement deux sous--sommes nulles, l'une de deux termes et l'autre de quatre termes} 
 
D'après ce que nous avons vu au \S$\ref{SSS:deuxiemecas}$.1, 
il existe une permutation $(j_1,j_2,j_3)=\bigl(\sigma(1),\sigma(i), \sigma(j)\bigr)$ telle que 
$$
\sgn(\sigma) \SSomme_1
=\alphatilde_{j_2}\beta_{j_1} +
\alphatilde_{j_3} \beta_{j_2}
\quad\hbox{et}\quad
\sgn(\sigma) \SSomme_2= 
 \beta_{j_3}\alphatilde_{j_1}
- \beta_{j_1}\alphatilde_{j_3}-
\alphatilde_{j_1}\beta_{j_2} - 
\alphatilde_{j_2}\beta_{j_3}
$$
pour laquelle aucune sous--somme stricte de $\SSomme_2$ ne s'annule, alors que $\SSomme_1=\SSomme_2=0$. 
On pose
$$ 
\delta_1
=\frac{ \alpha_{j_2}}{ \alpha_{j_1}},
\quad
\delta_2=\frac{\gamma_{j_1} \alpha_{j_3} }{\gamma_{j_3} \alpha_{j_1}},
\quad 
\delta_3=\frac{ \gamma_{j_2} }{\gamma_{j_3} },
\quad 
\udelta=(\delta_1,\delta_2,\delta_3), 
$$
de sorte que 
$$ 
\delta_1\frac{ \varepsilon_{j_2}}{ \varepsilon_{j_1}}
+ 
\delta_2\frac{w_{j_1} }{w_{j_3} }
\frac{ \varepsilon_{j_3} }{ \varepsilon_{j_1}}
+ \delta_3\frac{ w_{j_2} }{w_{j_3} }=1.
$$
Ainsi les $S$--unités
$$
\xi_1 := \frac{ \varepsilon_{j_2}}{ \varepsilon_{j_1}},\quad
\xi_2 := 
\frac{
w_{j_1} }
{w_{j_3} }
\frac{
 \varepsilon_{j_3} }
{ \varepsilon_{j_1}}
\quad\hbox{et}\quad
\xi_3 :=
\frac{w_{j_2} }{w_{j_3}}
$$
appartiennent à $\widetilde{A}_3^{(3)}(\udelta)$. La relation $\SSomme_1=0$ donne 
$$
\frac{w_{j_1} }{w_{j_2} }
\frac{ \varepsilon_{j_2} }{ \varepsilon_{j_3}}
=
-\frac{\gamma_{j_2} \alpha_{j_3} }{\gamma_{j_1} \alpha_{j_2}}
\cdotp
$$
On trouve alors
$$
\frac{ \varepsilon_{j_2}^2 }{ \varepsilon_{j_3}^2}=-\frac{\gamma_{j_2} \alpha_{j_3} }{\gamma_{j_1} \alpha_{j_2}}\xi_1\xi_2^{-1}\xi_3,
\quad
\frac{ \varepsilon_{j_1}^2 }{ \varepsilon_{j_3}^2}=-\frac{\gamma_{j_2} \alpha_{j_3} }{\gamma_{j_1} \alpha_{j_2}}\xi_1^{-1}\xi_2^{-1}\xi_3.
$$
 On pose donc
$$
 A_3(\ugamma,i)=\widetilde{A}_3^{(3)}(\udelta)\bigcup
\left\{
\zeta \in \OS^\times\; \mid \; \hbox{il existe $\vartheta_1, \vartheta_2, \vartheta_3 \in\widetilde{A}_3^{(3)}(\udelta)$ tels que 
$\zeta^{\pm 2}= - \frac{\displaystyle \vartheta_1 \vartheta_2 \vartheta_3 \gamma_{j_2} \alpha_{j_3} }{\displaystyle \gamma_{j_1} \alpha_{j_2}}$}
\right\}.
$$

 \bigskip

 \subsection{Fin de la démonstration}
Posons $u_1=1$, $u_2=\varepsilon_2/\varepsilon_1$ et $u_3=\varepsilon_3/\varepsilon_1$. On a donc $u_i\in  A_3(\ugamma,i)$ pour $i=1,2,3$. On définit $\alpha'_i:=\alpha_i u_i$, de sorte que 
$$
\alphatilde_i=\alpha_i\varepsilon_i=\varepsilon_1\alpha_i u_i=\varepsilon_1\alpha'_i.
$$
On pose encore $w=\varepsilon_1 y$. En éliminant $x$ et $w$ entre les trois équations
\begin{equation}\label{Equation:xw}
\beta_i=x-\alpha'_iw \quad (i=1,2,3),
\end{equation}
on trouve 
$$
\beta_1(\alpha'_2-\alpha'_3)+
\beta_2(\alpha'_3-\alpha'_1)+
\beta_3(\alpha'_1-\alpha'_2)=0,
$$
ce qui signifie que $(\beta_1,\beta_2,\beta_3)$ vérifie la propriété demandée à $(t_1,t_2,t_3)$ dans l'équation $(\ref{EquationAlphaPrime})$. Il en résulte que les éléments $v_i := \beta_i/\beta_1$ ($i=1,2,3$) appartiennent à $ A_4(\uu)$.

 \medskip
 Posons $\etatilde=w_1$, $\eta=\varepsilon_1$. On a donc $\beta_1=\denominateur^{-1}\gamma_1\etatilde$.
 Soit $i\in\{2,3\}$. On calcule $y=\eta^{-1}w$ en éliminant $x$ entre deux des équations $(\ref{Equation:xw})$ :
 $$
 y=\eta^{-1}\etatilde y_0,\quad \hbox{avec}\quad y_0=\frac{\gamma_1(v_1-v_i)}{\denominateur(\alpha'_i-\alpha'_1)}\cdotp
 $$
 On utilise encore une des équations $(\ref{Equation:xw})$ pour trouver $x$ :
 $$
 x= \etatilde x_0,\quad \hbox{avec}\quad x_0= {\denominateur}^{-1} \gamma_1+\alpha'_1 y_0.
 $$
 On pose enfin 
 $$
 \varepsilon_0=\frac{\gamma_1^3 v_2 v_3}{\denominateur^n \mu}, 
 $$
 de sorte que $\varepsilon=\etatilde^3\varepsilon_0$, ce qui montre que les solutions $(x,y,z,\varepsilon,\varepsilon_1,\varepsilon_2,\varepsilon_3)$ et 
 $(x_0,y_0,z_0\varepsilon_0,1, u_2, u_3)$ sont $S^3$--dépendantes.
 
 \begin{proof}[{}]
 À une solution $(x,y,z,\varepsilon,\varepsilon_1,\varepsilon_2,\varepsilon_3)$ de l'équation $(\ref{Equation:ThueGeneralisee})$, nous avons associé un élément 
$\uk'$ de $ A_2$, puis un élément 
$\ugamma$ de $A_1 (\uk')$ (rappelons qu'il y a au plus $\kappa_6$ éléments dans la réunion des $A_1 (\uk')$ quand $\uk'$ décrit $ A_2$), puis un élément $\uu$ de $\{1\}\times  A_3(\ugamma,2)\times \cdots \times  A_3(\ugamma,3)$ (il y en a au plus $\kappa_4^2$), et enfin un élément $\uv$ de $ A_4(\uu)^2$ (il y en a au plus $\kappa_3^2$). Nous en avons conclu que la solution de départ était dans la même classe de $S^3$--dépendance que $(x_0,y_0,z_0\varepsilon_0,1, u_2, u_3)$. Le théorème $\ref{Theoreme:Principal}$ en résulte, avec 
\begin{equation}
\label{Equation:cstTheoreme}
\kappa_1=\kappa_6 
\kappa_3^2
\kappa_4^2. 
\end{equation} 
Comme annoncé, cette constante $\kappa_1$ ne dépend que de
$K$, $\rmN_S(\mu)$, $\alpha_1, \alpha_2,\alpha_3$ et du rang $s$ du groupe $\OS^\times$. Rappelons que 
$$
\kappa_3=1+ 2^{8s+24},
\quad
\kappa_4=1+ (2^{4375}3^{250})^s,
\quad
\kappa_6=d(m)^2
\kappa_5^3 m
$$
avec 
$m=\rmN_S(\denominateur^3\mu)$, où
 $\denominateur$ est le plus petit entier rationnel positif tel que $\denominateur\alpha_i\in\OS$ pour $i=1,2,3$,
$$
\kappa_5=2^{r +1}\pi^{r_2} |D_K|^{-1/2} e^{c_3 dR_K} \nu^{tdh_K} (1+\theta)^{d},
$$
où $r$ est le rang du groupe des unités de $K$, $d$ son degré,
$D_K$ son discriminant, $h_K$ son nombre de classes, $R_K$ son régulateur, $\nu$ est le plus grand des nombres $|\rmN(\gothP_i)|$ ($1\le i\le t$) quand $\gothP_1,\ldots,\gothP_t$ sont les idéaux premiers de $\OK$ au--dessus des places ultramétriques de $K$ dans $S$,
$$
\theta\le |D_K|^{1/2}, \quad c_3=\frac{1}{2}r^{r+1} \delta_K^{-(r-1)}
$$
et $\delta_K$ est la constante (de Blanksby et Mongomery, Dobrowolski et Voutier, {\it cf.} \cite{MR1373714}) 
telle que tout entier algébrique non nul $\alpha\in K$ qui n'est pas une racine de l'unité vérifie $\rmh(\alpha)\ge \delta_K /d$. 

 \end{proof}

\section{Preuve du corollaire $\ref{Corollaire:Principal}$} \label{S:PreuveCorollaire}
\begin{proof}[{}]
On utilise le théorème $\ref{Theoreme:Principal}$ avec 
$$
n=d, \quad
\mu=k/a_0,\quad
\alpha_i=\sigma_i(\alpha)\quad
\hbox{et}\quad
\varepsilon_i=\sigma_i(\varepsilon) \quad \hbox{pour $i=1,\dots, d$}.
$$
L'hypothèse $ [\Q(\alpha \varepsilon):\Q]\ge 3 $ du corollaire $\ref{Corollaire:Principal}$ garantit que, quitte à permuter les $\alpha_i$, l'hypothèse $\Card\{\alpha_1 \varepsilon_1, \alpha_2\varepsilon_2,\alpha_3\varepsilon_3 \}\ge 3$ du théorème $\ref{Theoreme:Principal}$ est satisfaite.

Soient $(x,y,\varepsilon,z_1,\ldots,z_t)$ et $(x',y',\varepsilon',z'_1,\ldots,z'_t)$ deux solutions dans $\Z^2 \times \OS^\times\times \N^t$ et  $S^3$--dépendantes de l'équation $(\ref{Equation:CorollairePrincipal})$, satisfaisant $\pgcd(xy,p_1\cdots p_t)=\pgcd(x'y',p_1\cdots p_t)=1$. La condition de $S^3$--dépendance signifie qu'il 
existe des $S$--unités $\eta$ et $\etatilde$ dans $\OS^\times$ telles que 
$$ 
x'=x\etatilde,\; y'=y\eta^{-1}\etatilde, \; p_1^{z'_1}\cdots p_t^{z'_t}= p_1^{z_1}\cdots p_t^{z_t}\etatilde^n, \; \sigma_i(\varepsilon')=\sigma_i(\varepsilon)\eta \quad(1\le i\le n).
$$ 
La relation $x'=x\etatilde$ implique que $\etatilde\in\Q\cap \OS^\times=\{\pm 1\}\times p_1^\Z\times \cdots \times p_t^\Z$. Comme $x$ et $x'$ sont premiers avec $p_1\cdots p_t$, on en déduit 
$\etatilde=\pm 1$. Ensuite la relation $y'=y\eta^{-1}\etatilde$ implique, avec les mêmes arguments, que $\eta=\pm 1$. Il en résulte que l'on a $\varepsilon'=\eta\varepsilon$ et $(z'_1,\ldots,z'_t)=(z_1,\ldots,z_t)$. 
Par conséquent, dans chaque classe de $S^3$--dépendance, 
 il y a au plus quatre éléments $(x,y,\varepsilon,z_1,\ldots,z_t)$ satisfaisant $f_\varepsilon (x,y)=k p_1^{z_1}\cdots p_t^{z_t}$. 
 On en déduit le corollaire $\ref{Corollaire:Principal}$ avec 
 \begin{equation}\label{Equation:cstCor}
\kappa_2=4 \kappa_1,
\end{equation} 
où on prend pour paramètres de la constante $\kappa_1$ le corps de nombres $K=\Q(\alpha)$, puis $\mu=k$, $\alpha_i=\sigma_i(\alpha)$ ($i=1,2,3$) et $s=t+r+1$, où $r$ est le rang du groupe des unités de $K$. 
 
\end{proof}

 \section{Preuve de la proposition $\ref{Proposition:InfiniteClassesEquivalence}$} \label{S:DemonstrationProposition}
 
 \begin{proof}[{}]
Il s'agit de vérifier que pour tout $\uvarepsilon\in\calE\setminus \calE^\star$, il n'y a qu'un nombre fini de $\uvarepsilon'\in\calE\setminus \calE^\star$ tels que les formes $f_{\uvarepsilon}$ et $f_{\uvarepsilon'}$ soient $S$--équivalentes. 
 Soient donc $\uvarepsilon$ et $\uvarepsilon'$ deux éléments distincts de $\calE\setminus \calE^\star$ tels que $f_{\uvarepsilon}$ et $f_{\uvarepsilon'}$ soient $S$--équivalents~: il existe $(\alpha, \beta, \gamma, \delta, \eta)\in \OS^4\times \OS^\times$ tel que
 $\alpha\delta-\beta\gamma \in\OS^\times$ et 
$$ 
 f_{\uvarepsilon}(X,Y) = \eta 
f_{\uvarepsilon'} (\alpha X+\beta Y, \gamma X+\delta Y).
$$ 
Cette relation 
s'écrit
$$
\prod_{j=1}^n (X- \alpha_j\varepsilon_j Y)=\eta \prod_{i=1}^n \bigl( (\alpha- \gamma \alpha_i \varepsilon'_i)X+(\beta- \delta\alpha_i \varepsilon'_i )Y\bigr),
$$
c'est-à-dire
$$
\eta \prod_{i=1}^n (\alpha- \gamma\alpha_i \varepsilon'_i )=1
\quad\hbox{et}\quad 
\prod_{j=1}^n (X- \alpha_j\varepsilon_j Y) = \prod_{i=1}^{n} \left( X + \frac{\beta - \delta \alpha_i \varepsilon'_i}{\alpha - \gamma\alpha_i\varepsilon'_i} Y \right).
$$
Donc il existe une permutation $\sigma$ de $\{1,2,\dots,n\}$ pour laquelle on a, pour $i=1,\dots,n$,
 $$
 \alpha_{\sigma(i)} \varepsilon_{\sigma(i)} 
(\gamma \alpha_i \varepsilon'_i -\alpha)=\beta-\delta\alpha_i \varepsilon'_i.
 $$ 
Soit $i$ un indice dans $\{2,\ldots,n\}$ tel que $\alpha_i\varepsilon'_i\not=\alpha_1$ (rappelons que $\varepsilon'_1=1$). Nous allons vérifier que $\varepsilon'_i$ appartient à un ensemble fini indépendant de $(\alpha, \beta, \gamma, \delta, \eta)$. 
Soit $j$ un autre indice dans $\{2,\ldots,n\}$ tel que les trois nombres 
$\alpha_1, \alpha_i\varepsilon'_i, \alpha_j\varepsilon'_j$ soient deux à deux distincts. 
 
Comme $ f_{\uvarepsilon}(1,0)=1$ et $f_{\uvarepsilon'}(\alpha\delta-\beta\gamma,0)=(\alpha\delta-\beta\gamma)^n$, on a 
$$
f_{\uvarepsilon'}(\alpha,\gamma) =\eta^{-1} \in\OS^\times
 \quad\hbox{et}\quad
f_{\uvarepsilon}(\delta,-\gamma) =\eta (\alpha\delta-\beta\gamma)^n \in\OS^\times.
$$
Par définition de $\calE^\star$, il en résulte que l'on a $\alpha\gamma=\delta\gamma=0$. 
 
 Si $(\alpha,\delta)=(0,0)$, on a $\beta\gamma\not=0$ et 
 $$
 \alpha_{\sigma(i)} \varepsilon_{\sigma(i)} 
\alpha_i \varepsilon'_i =
\alpha_{\sigma(1)} \varepsilon_{\sigma(1)} 
\alpha_1,
 $$
ce qui montre qu'il y a au plus $n!$ valeurs possibles pour $\varepsilon'_i$. 

Il reste à considérer le cas où $\gamma=0$. 
Les trois équations
 $$ 
 \begin{cases}
&\hbox{$\alpha \alpha_{\sigma(1)} \varepsilon_{\sigma(1)} 
- \delta\alpha_1 +\beta=0$}
\\
&\hbox{$\alpha \alpha_{\sigma(i)} \varepsilon_{\sigma(i)} 
 - \delta\alpha_i \varepsilon'_i +\beta=0$}
\\
&\hbox{$\alpha \alpha_{\sigma(j)} \varepsilon_{\sigma(j)} 
- \delta\alpha_j \varepsilon'_j +\beta=0$}
\end{cases}
 $$
montrent que $\alpha_{\sigma(1)}\varepsilon_{\sigma(1)}$, $\alpha_{\sigma(i)}\varepsilon_{\sigma(i)}$ et $\alpha_{\sigma(j)}\varepsilon_{\sigma(j)}$ sont deux à deux distincts. 
 On élimine $\alpha$, $\beta$ et $\delta$ entre ces trois équations, 
 ce qui donne l'équation aux $S$--unités {\it à la Siegel}
 $$
(\alpha_{\sigma(i)} \varepsilon_{\sigma(i)}-\alpha_{\sigma(j)} \varepsilon_{\sigma(j)}) \alpha_1 
+(\alpha_{\sigma(j)} \varepsilon_{\sigma(j)}-\alpha_{\sigma(1)} \varepsilon_{\sigma(1)}) \alpha_i \varepsilon'_i 
+(\alpha_{\sigma(1)} \varepsilon_{\sigma(1)}-\alpha_{\sigma(i)} \varepsilon_{\sigma(i)}) \alpha_j \varepsilon'_j =0.
$$
On déduit du théorème $\ref{Theoreme:Evertse}$ que $\varepsilon'_i$ appartient à un ensemble fini, ce qu'il fallait démontrer. De plus, comme cette dernière équation aux unités n'a que trois termes, le résultat peut être démontré de façon effective: on peut donner une borne pour les hauteurs des $\uvarepsilon'\in\calE\setminus\calE^\star$ tels que les formes $f_{\uvarepsilon}$ et $f_{\uvarepsilon'}$ soient $S$--équivalentes. 

\end{proof}


\subsection*{Remerciements} 
\begin{small}
Ce projet de recherche est né lors de séances de jogging sur le bord des plages de Rio.
Les auteurs expriment leur reconnaissance à ceux qui ont rendu possibles, agréables et productifs leurs séjours communs au Brésil (Instituto Nacional de Matematica Pura e Applicada (IMPA)), au Népal (Université de Tribhuvan et Université de Kathmandu) et en Inde (Congrès International des Mathématiciens ICM 2010)
où ces recherches ont été entreprises et concrétisées.
De pertinentes remarques d'Isao Wakabayashi sur une version préliminaire de ce texte nous ont été utiles. 
Gaël Rémond nous a aidés à faire le lien (qui nous avait été indiqué par un arbitre anonyme) avec le théorème de Vojta. 
\end{small}

\bigskip
\goodbreak


\end{document}